\documentclass[a4paper, dvipdfmx, 12pt]{amsproc}

\usepackage{fullpage}
\usepackage{amsmath, amsthm, amssymb, mathtools, mathrsfs}

\usepackage{mathtools}

\usepackage{hyperref}

\usepackage{chessfss}
\usepackage{tikz-cd}
\usepackage{tikz}
\usetikzlibrary{intersections, calc, arrows.meta}

\usepackage{graphicx}
\usepackage{here}
\usepackage{time}

\usepackage{xcolor}
\usepackage[capitalize,nameinlink,noabbrev,nosort]{cleveref}
\hypersetup{
	colorlinks=true,       
	linkcolor=blue,          
	citecolor=red,        
	filecolor=brown,      
	urlcolor=brown,           
}

\makeatletter
\@namedef{subjclassname@2020}{%
  \textup{2020} Mathematics Subject Classification}
\makeatother

\newtheorem{theoremcounter}{Theorem Counter}[section]

\theoremstyle{definition}
\newtheorem{dfn}[theoremcounter]{Definition}
\newtheorem{remark}[theoremcounter]{Remark}

\theoremstyle{plain}
\newtheorem{lem}[theoremcounter]{Lemma}
\newtheorem{proposition}[theoremcounter]{Proposition}

\newtheorem{thm}[theoremcounter]{Theorem}

\numberwithin{equation}{section}

\newcommand{\Ker}{\operatorname{Ker}}

\begin{document}

\title[]{The $p$-adic constant for mock modular forms associated to CM forms I\hspace{-1.2pt}I} 

\author{Ryota Tajima} 
\address{Joint Graduate School of Mathematics for Innovation, Kyushu University, 744 Motooka, Nishi-ku, Fukuoka, 819-0395, Japan}
\email{ryota.tajima.123@gmail.com}

\thanks{The author was supported by JSPS KAKENHI Grant Number JP23KJ1720 and WISE program (MEXT) at Kyushu University. 
}



\maketitle

\begin{abstract}
For a normalized newform $g \in S_{k}(\Gamma_{0}(N))$ with complex multiplication by an imaginary quadratic field $K$,  there is a mock modular form $F^{+}$ corresponding to $g$. K. Bringmann et al. \cite{bringmann2012mock} modified $F^{+}$  in order to obtain a $p$-adic modular form by a certain $p$-adic constant $\alpha_{g}$. In addition, they \cite{bringmann2012mock} showed that if $p$ is split in $\mathcal{O}_{K}$ and $p \nmid N$, then $\alpha_{g}=0$. On the other hand, the author \cite{tajima2023p} showed that $\alpha_{g}$ is a $p$-adic unit for an inert prime $p$ satisfying that $p\nmid 2N$ when $\dim_{\mathbb{C}} S_{k}(\Gamma_{0}(N))=1$. In this paper, under mild condition, we determine the $p$-adic valuation of $\alpha_{g}$ for an inert prime $p$ and a general CM form $g$ of weight $2$ with rational Fourier coefficients.
\end{abstract}

\section{Introduction}
Ramanujan found an interesting and mysterious $q$-series, mock theta function. These $q$-series are not modular forms, but they behave like modular forms. Many mathematicians tried to investigate their properties and reveal the source of their properties, which are similar to modular forms. In 2002, S. Zwegers \cite{zwegers2008mock} figured out the origin of their properties. He showed that Ramanujan's mock theta function is the holomorphic part of some harmonic Maass form $F$. The holomorphic part of a harmonic Maass form $F^{+}$ is called mock modular form.

Fourier coefficients of a mock modular form are transcendental in general. However, for a normalized newform $g$, we can associate mock modular forms that have algebraic properties. It is known that there is a surjective map $\xi _{2-k}:=2iy^{2-k}\overline{( \frac{\partial }{\partial \overline{\tau}}) }$ from harmonic Maass forms to cusp forms, and we consider lifts of $g$ by $\xi_{2-k}$. J. H. Bruinier, K. Ono, and R. C. Rhoades \cite{bruinier2008differential} found lifts of $g$ satisfying a certain algebraic properties. These lifts are called good for $g$. (cf. Definition  \ref{good}.) For example, if $g$ has complex multiplication and $F$ is good for $g$, then all Fourier coefficients of $F^{+}$ are algebraic.

Let $g \in S_{k}(\Gamma_{0}(N))$ be a normalized newform with complex multiplication by an imaginary quadratic field $K$ and $F$ be good lift of $g$. Since all Fourier coefficients of $F^{+}$ are algebraic, we could regard $F^{+}$ as an element in $\mathbb{C}_{p}(\!(q)\!)$ and investigate $p$-adic properties of $F^{+}$. K. Bringmann, P. Guerzhoy, and B. Kane \cite{bringmann2012mock} found the $p$-adic number $\alpha_{g}$ and constructed a $p$-adic modular form by using $\alpha_{g}$. We explain their results when $p$ is inert in $\mathcal{O}_{K}$ and $p \nmid N$. We denote the $n$-th Fourier coefficients of $g$ and $F^{+}$ by $a_{g}(n)$ and $a^{+}(n)$. For any $p$-adic number $\gamma \in \mathbb{C}_{p}$, the formal power series $\mathcal{F}_{\gamma}(q) \in \mathbb{C}_{p}(\!(q)\!)$ is defined by
\begin{align*}
\mathcal{F}_{\gamma}(q):=F^{+}(q)-\gamma E_{g \mid V_{p}}(q)= \sum _{n\gg -\infty}\left( a^{+}( n) -\gamma n^{1-k}a_{g}(n/p) \right) q^{n}
\end{align*}
where $E_{g \mid V_{p}}(q)=\sum_{n >0}n^{1-k}a_{g}(n/p)q^{n}$. Although this $q$-series $\mathcal{F}_{\gamma}(q)$ is just a formal power series, K. Bringmann et al. \cite{bringmann2012mock} showed that there exists a unique $p$-adic number $\alpha_{g}$ such that $\mathcal{F}_{\alpha_{g}}$ is a Fourier expansion of a $p$-adic modular form. In addition, L. Candelori and F. Castella \cite{candelori2017geometric} showed that $\mathcal{F}_{\alpha_{g}}$ is an overconvergent modular form. We think that investigating the $p$-adic number $\alpha_{g}$ is a clue to figure out $p$-adic properties of mock modular forms. K. Bringmann et al. \cite{bringmann2012mock} showed that $\alpha_{g}=0$ when $p$ is split in $\mathcal{O}_{K}$ and $p\nmid N$. On the other hand, the author \cite{tajima2023p} showed that $\alpha_{g}$ is a $p$-adic unit when $p$ is inert in $\mathcal{O}_{K}$, $p \nmid 2N$, and $\dim_{\mathbb{C}} S_{k}(\Gamma_{0}(N))=1$. (cf. \cite[Theorem 1.5]{tajima2023p}.) In this paper, we eliminate the assumption on the dimension of the space $S_{k}(\Gamma_{0}(N))$ and determine the $p$-adic valuation of $\alpha_{g}$ when $g$ is of weight $2$ with rational Fourier coefficients. Especially, we show that the $p$-adic valuation of $\alpha_{g}$ is related to the degree of a modular parametrization.

We explain the main result. We define the period lattice of $g$ and an elliptic curve $E/ \mathbb{Q}$ by 
\begin{align*}
\Lambda:=\left\{ \int ^{\gamma(\infty)}_{\infty}2\pi ig(\tau) d\tau \mid \gamma \in \Gamma _{0}(N) \right\},
\end{align*} 
\begin{align}\label{daen shiki}
E : y^{2}=4x^{3}-g_{2}(\Lambda)x-g_{3}(\Lambda).
\end{align}
It is known that $c_{4}:=12g_{2}(\Lambda)$ and $c_{6}:=216g_{3}(\Lambda)$ are rational integers. (cf.  \cite[Section 2.14]{cremona1997algorithms}.) Hence, the Weierstrass model \eqref{daen shiki} is defined over $\mathbb{Z}[1/6]$.
There is a modular parameterization $\phi$ from the modular curve $X_{0}(N)$ to $E$ defined by
\begin{align*}
\phi : \tau \mapsto -\int^{i\infty}_{\tau} 2\pi ig(\tau')d\tau' \text{ mod } \Lambda,
\end{align*}
and we denote the degree of $\phi$ by $C_{g}$.
Our main result is as follows.

\begin{thm}
Let $g \in S_{2}(\Gamma_{0}(N))$ be a normalized newform with complex multiplication by an imaginary quadratic field $K$ with rational Fourier coefficients. We assume that $p \geq 5$ and the Weierstrass model \eqref{daen shiki} has good reduction at $p$.
If $p$ is inert in $\mathcal{O}_{K}$, then $C_{g}\alpha_{g}$ is a $p$-adic unit. 
\end{thm}

We explain the outline of the proof and the reason why we could eliminate the assumption on the dimension of the space $S_{k}(\Gamma_{0}(N))$ in \cite{tajima2023p}. K. Bringmann et al. \cite{bringmann2012mock} showed that the $p$-adic number $\alpha_{g}$ is written by 
\begin{align}\label{eq. alpha_g}
\alpha_{g}=\lim _{m \rightarrow \infty }\frac{p^{2m+1}a^{+}(p^{2m+1}) }{(-p)^{m}}.
\end{align} 
In \cite[Theorem 1.5]{tajima2023p}, we use the equation \eqref{eq. alpha_g} in order to show that $\alpha_{g}$ is a $p$-adic unit. Therefore, the proof of \cite[Theorem 1.5]{tajima2023p} is based on an explicit calculation of the $p$-adic valuation of $a^{+}(p^{2m+1})$.
In this paper, we give a new characterization of $\alpha_{g}$ that does not use Fourier coefficients. Let $E_{g}(q):=\sum_{n>0}\left(a_{g}(n)/n\right)q^{n}$ be the Eichler integral of $g$. P. Guerzhoy \cite{guerzhoy2014zagier} showed that there exists unique $p$-adic numbers $\lambda_{p}$ and $\mu_{p}$ such that
\begin{align}\label{eq. imp}
F^{+}(q)-\lambda_{p}E_{g}(q)-\mu_{p}E_{g\mid V_{p}}(q) \in \mathbb{Z}_{p}(\!(q)\!) \otimes \mathbb{Q}_{p}
\end{align}
holds.  (cf. \cite[Proposition 5]{guerzhoy2014zagier}.) He is interested in the $p$-adic number $\lambda_{p}$ because of Zagier's observation.
In this paper, by contrast, we focus on the $p$-adic number $\mu_{p}$ and show that $\alpha_{g}$ is equal to $\mu_{p}$. By Honda theory, there exists the formal group $\mathcal{G}$ defined over $\mathbb{Z}_{p}$ whose logarithm is equal to $E_{g}(q)$. It is known that the de Rahm cohomology $H_{\text{dR}}^{1}(\mathcal{G})$ is a free $\mathbb{Z}_{p}$ module of rank $2$ because $p$ is inert in $\mathcal{O}_{K}$. The equation \eqref{eq. imp} gives a sort of $p$-adic Hodge decomposition of $F^{+}(q)$ in $H_{\text{dR}}^{1}(\mathcal{G})$. On the other hand, P. Guerzhoy \cite{guerzhoy2014zagier} and C. Alfes, M. Griffin, K. Ono, and L. Rolen \cite{alfes2015weierstrass} defined a $q$-series $\mathcal{N}(q)=\mathcal{N}^{+}(q)+\mathcal{N}^{-}(q)$ by
\begin{align*}
\mathcal{N}^{+}(q) := \zeta(\Lambda, E_{g}(q))-\mathbb{S}(\Lambda)E_{g}(q), \quad \mathcal{N}^{-}(q):=- \dfrac{\pi}{a(\Lambda)}\overline{E_{g}(q)}
\end{align*} 
where $\mathbb{S}(\Lambda)$ is the Eisenstein number of weight $2$ and $a(\Lambda)$ is the area of the fundamental parallelogram of $\Lambda$. (cf. Section \ref{sec p-adic mock 0}.) We note that the Eisenstein number $\mathbb{S}(\Lambda)$ is a rational number because $E$ has complex multiplication. $\mathcal{N}(q)$ has a certain harmonic property.
C. Alfes et al. \cite{alfes2015weierstrass} called $\mathcal{N}^{+}(q)$ the Weierstrass mock modular form. In addition, P. Guerzhoy \cite{guerzhoy2014zagier} showed that 
\begin{align}\label{eq. 1}
\mathcal{N}^{+}(q)-C_{g}F^{+}(q) \in \mathbb{Z}_{p}(\!(q)\!) \otimes \mathbb{Q}_{p}
\end{align}
for almost all prime number $p$, and we show that the equation \eqref{eq. 1} holds for all inert prime number $p \nmid N$. (cf. \cite[Proposition 2]{guerzhoy2014zagier}.)
Therefore, we investigate Weierstrass mock modular form instead of $F^{+}(q)$. By using the $p$-adic decomposition of Weierstrass $\zeta$-function, we could show that $C_{g}\alpha_{g}$ is a $p$-adic unit.

\section*{Acknowledgements}
The author shows my most significant appreciation to Professor Shinichi Kobayashi for giving numerous and extremely helpful comments on this paper. He checked my draft many times, and the author had a lot of discussions with him about my paper. He also recommended me to read his paper \cite{bannai2017radius}. I also would like to thank Dr. Keiichiro Nomoto. We had many discussions, and he gave me valuable comments about calculators. He helped me to compute some examples of the period lattice by using mathematical software. I was supported by JSPS KAKENHI Grant Number JP23KJ1720 and WISE program (MEXT) at Kyushu University.

 \section{Harmonic Maass forms and weakly holomorphic modular forms}
 In this section, we introduce basic facts for harmonic Maass forms. We refer to \cite{bruinier2008differential} \cite{bringmann2017harmonic} as main references. 
 
Throughout, we denote the upper-half of the complex plane by $\mathbb{H}$.

\begin{dfn}[\cite{bruinier2008differential}]
Let $k \in \mathbb{Z}$ and $N \in \mathbb{N}$.
If a $C^{\infty}$ function $F$ from $\mathbb{H}$ to $\mathbb{C}$ satisfies under conditions, then we say that $F$ is a harmonic Maass form of weight $k$ on $ \Gamma _{0}(N)$. 

(1)$F\left(\dfrac{a\tau+b}{c\tau+d}\right)=(c\tau+d)^{k}F(\tau)$ for all $\gamma =  \bigl(
\begin{smallmatrix}
   a & b \\
   c & d
\end{smallmatrix}
\bigl) \in \Gamma _{0}( N)$.

(2)$\Delta_{k}F=0$, where
 $\Delta_{k}=-y^{2}\left( \dfrac{\partial ^{2}}{\partial x^{2}}+\dfrac{\partial ^{2}}{\partial y^{2}}\right) +iky\left( \dfrac{\partial }{\partial x}+i\dfrac{\partial }{\partial y}\right) $. 
 
(3)There is a polynomial $P_{\infty}(z) \in \mathbb{C}[q^{-1}]$ such that
\begin{align*}
F(\tau)-P_{\infty}(\tau)=\text{O} (e^{-\varepsilon y})  \text{ as } y  \rightarrow \infty \text{ for some } \varepsilon >0.
\end{align*}
Similar conditions are required at other cusps.
 \end{dfn}

We denote the space of harmonic Maass forms by $H_{k}( \Gamma _{0}( N) )$.

It is known that a harmonic Maass form $F(\tau)$ of weight $2-k$ has a Fourier expansion of the form
\begin{align*}
F(\tau)=\sum _{n\gg -\infty}a^{+}( n) q^{n}+\sum _{n <0}a^{-}( n) \Gamma ( k-1,4\pi  \left| n\right| y) q^{n} 
\end{align*}
where  $\Gamma(s, z):=\int ^{\infty }_{z}e^{-t}t^{s}\dfrac{dt}{t}$.
From this Fourier expansion, we define the holomorphic part $F^{+}(\tau)$ and the non-holomorphic part $F^{-}(\tau)$ by

\begin{align*}
F^{+}(\tau):=\sum _{n\gg -\infty}a^{+}( n) q^{n},
\end{align*}

\begin{align*}
F^{-}(\tau):=\sum _{n <0}a^{-}( n) \Gamma ( k-1,4\pi  \left| n\right| y) q^{n}.
\end{align*}

In addition, the partial sum $\sum _{n \leq 0}a^{+}( n) q^{n}$ is called the principal part of $F(\tau)$ at the cusp $\infty$.

Especially, we have that $\Gamma (k-1,4\pi  \left| n\right| y)=e^{-4\pi \left| n \right| y}$ when $k=2$. Hence if $F(\tau) \in H_{0}( \Gamma _{0}( N) )$, then
\begin{align*}
F^{-}(\tau)=\sum_{n>0}a^{-}(-n)\bar{q}^{n}
\end{align*} 
holds where $\bar{q}$ is the complex conjugate of $q$.

\begin{dfn}[{\cite[Definition 5.16]{bringmann2017harmonic}}]
We define a mock modular form by the holomorphic part of a harmonic Maass form whose non-holomorphic part is non-trivial.
\end{dfn}

 Next, we introduce two important differential operators $D^{k-1}$ and $\xi_{2-k}$. We denote the space of weakly holomorphic modular forms of weight $k$ and level $N$ by $M_{k}^{!}( \Gamma _{0}( N) )$.  For a ring $R$, we denote the space of cusp forms (resp. weakly holomorphic modular forms) of weight $k$ and level $N$ whose all Fourier coefficients at $\infty$ are in $R$ by $S_{k}(\Gamma_{0}(N), R)$ (resp. $M_{k}^{!}(\Gamma_{0}(N), R)$). For $k \in \mathbb{Z}_{\geq 2}$, we define two operators $D:=\frac{1}{2 \pi i}\frac{d}{d\tau}$ and $\xi _{2-k}:=2iy^{2-k}\overline{( \frac{\partial }{\partial \overline{\tau}}) }$. Then, these operators give maps
 \begin{align*}
& D^{k-1}:H_{2-k} ( \Gamma _{0}( N) ) \rightarrow M_{k}^{!}( \Gamma _{0}( N) ), 
\\
 &\xi_{2-k} : H_{2-k} ( \Gamma _{0}( N) ) \rightarrow S_{k}( \Gamma _{0}( N) ), 
 \end{align*}
 and we have that
 \begin{align*}
 D^{k-1}(F^{-})=0, \quad \xi_{2-k}(F^{+})=0
 \end{align*}
 for $F \in H_{2-k}( \Gamma _{0}(N))$.
 
 \begin{thm}[{\cite[Corollary 3.8]{bruinier2004two}}]\label{xi}
The map $\xi_{2-k}:H_{2-k} ( \Gamma _{0}( N) ) \rightarrow S_{k}( \Gamma _{0}( N) ) $ is subjective, and $\Ker \xi_{2-k}=M_{2-k}^{!}(\Gamma_{0}(N))$ holds.
 \end{thm}

The image of $F$ by $\xi_{2-k}$ is called the shadow of $F^{+}$.

 \begin{dfn}[{\cite[Definition 7.5]{bringmann2017harmonic}}]\label{good}
 Let $g \in S_{k}( \Gamma _{0}( N) )$ be a normalized newform and $F_{g}$ be the Hecke field of $g$.
If a harmonic Maass form $F \in H_{2-k} ( \Gamma _{0}( N) )$ satisfies under conditions, we say that $F$ is good for $g$.

 (1)The principal part of $F$ at the $\infty$ belongs to $F_{g}[q^{-1}]$.
 
 (2)The principal part of $F$ at the other cusps of $\Gamma_{0}(N)$ are constant.
 
 (3)We have that $\xi _{2-k}(F) =\dfrac{g}{\left\| g\right\| ^{2}}$.
  \end{dfn}

 \begin{proposition}[{\cite[Proposition 5.1]{bruinier2008differential}}]
 If $g \in S_{k}(\Gamma_{0}(N))$ is  a normalized newform, then there is a harmonic Maass form $F \in H_{2-k}(\Gamma_{0}(N))$ which is good for $g$. 
 \end{proposition}
 
 \begin{remark}
 For a normalized newform $g$, a harmonic Maass form that is good for $g$ is not unique. 
 \end{remark}

 Although the definition of good for $g$ requires the algebraicity of the principal part, it is known that all Fourier coefficients of $F^{+}$ are algebraic when $g$ has complex multiplication.
 
 \begin{thm}[{\cite[Corollary 1.2]{ehlen2024harmonic}}]\label{Thm alg of cof}
 Let $g \in S_{k}( \Gamma _{0}( N) )$ be a normalized newform with complex multiplication. 
 If $F \in H_{2-k} ( \Gamma _{0}( N) )$ is good for $g$, 
 then all Fourier coefficients of $F^{+}$ are in $F_{g}$.
 \end{thm}

\section{The $p$-adic properties of mock modular forms}

In this section, we recall the $p$-adic properties of mock modular forms whose shadow $g$ has complex multiplication by an imaginary quadratic field $K$. We refer to \cite{bringmann2012mock} \cite{candelori2017geometric} \cite{guerzhoy2014zagier}  \cite{guerzhoy2010p} as main references.

 From now on, we fix an algebraic closure $\overline{\mathbb{Q}}_{p}$ and embedding $\iota \colon \overline{\mathbb{Q}}\rightarrow \overline{\mathbb{Q}}_{p}$. We denote the $p$-adic completion by $\mathbb{C}_{p}$ and normalize the $p$-adic valuation $v_{p}$ such that $v_{p}(p)=1$.

For two formal power series 
\begin{align*}
A(q)=\sum_{n\geq -t} a(n)q^{n}, \quad B(q)=\sum_{n\geq -t} b(n)q^{n} \in \mathbb{C}_{p}(\!(q)\!),
\end{align*}
we denote $A(q) \equiv B(q) \text{ (mod }p^{m})$ when 
\begin{align*}
\min_{n \geq -t} v_{p}(a(n)-b(n)) \geq m.
\end{align*}
If $H(q)=\sum_{n\geq -t} a(n)q^{n} \in \mathbb{C}_{p}(\!(q)\!)$ satisfies under condition, then we say that $H(q)$ is a $p$-adic modular form of weight $k$ and level $N$.
For every $m \in \mathbb{N}$, there exists 
\begin{align*}
H_{m}(q)=\sum_{n\geq -t} a_{m}(n)q^{n} \in M_{k}^{!}( \Gamma _{0}(N), \bar{\mathbb{Q}})
\end{align*}
such that
\begin{align*}
H(q) \equiv H_{m}(q) \quad(\text{mod }p^{m})
\end{align*}
holds.
We define three operators $U_{p}, V_{p}$ and $T_{k, p}$ on $\mathbb{C}_{p}(\!(q)\!)$ by
\begin{align*}
& U_{p}\left(\sum_{n \in \mathbb{Z}} a(n)q^{n}\right):=\sum_{n \in \mathbb{Z}} a(pn)q^{n}
\\
& V_{p}\left(\sum_{n \in \mathbb{Z}} a(n)q^{n}\right):=\sum_{n \in \mathbb{Z}} a(n/p)q^{n}
\\
& T_{k, p}:=U_{p}+p^{k-1}V_{p}.
\end{align*}

Let $g \in S_{k}(\Gamma_{0}(N), \mathbb{Z})$ be a normalized newform with complex multiplication by an imaginary quadratic field $K$ and $F \in H_{2-k}(\Gamma_{0}(N))$ be good for $g$.  We define the Eichler integral of $g$ and $g \mid V_{p}$ by
\begin{align*}
E_{g}(q) :=\sum _{n>0}n^{1-k}a_{g}(n) q^{n}
\\
E_{g \mid V_{p}}(q) :=\sum _{n>0}(pn)^{1-k}a_{g}(n) q^{pn}.
\end{align*}
where $a_{g}(n)$ is the $n$-th coefficient of $g$.
For $\gamma \in \mathbb{C}_{p}$, we define the formal power series
\begin{align*}
\widetilde{\mathcal{F}}_{\gamma}:=F^{+}-\gamma E_{g|V_{p}}.
\end{align*}
Let $\beta, \beta'$ be the roots of $X^{2}-a_{g}(p)X+p^{k-1}$ satisfying $v_{p}( \beta ) \leq v_{p}( \beta ') $. 
We denote the $n$-th Fourier coefficients of $D^{k-1}(F)$ by $a_{D^{k-1}(F)}(n)$.

\begin{lem}[{\cite[Proposition 2.3]{guerzhoy2010p}}]\label{lemma a}
If $p$ is inert in $\mathcal{O}_{K}$ and $p \nmid N$, then 
\begin{align*}
\lim _{m\rightarrow \infty }\dfrac{a_{D^{k-1}(F)}( p^{2m+1}) }{\beta ^{2m}}
\end{align*}
is convergent.

\end{lem}

Next we calculate the value of $\lim _{m\rightarrow \infty }\dfrac{a_{D^{k-1}(F)}( p^{2m}) }{\beta ^{2m}}$ for an inert prime $p$. 

\begin{lem}[{\cite[Lemma 4.1]{bringmann2012mock}}]\label{lemma b}
Let $\chi_{K}$ be the character associated to $K$. The twist 
\begin{align*}
R := \dfrac{1}{2}\left(F+F\otimes \chi_{K}\right)\otimes \chi_{K}
\end{align*}
is an element of $M_{k}^{!}(\Gamma_{0}(N), \mathbb{Q})$.
\end{lem}

\begin{remark}
The Fourier coefficient $a_{R}(n)$ of $R$ is equal to $0$ when $\chi_{K}(n) \not= 1$. In particular, $a_{R}(n)=0$ when $\chi_{K}(n)=0$.
\end{remark}

\begin{lem}[{\cite[Proposition 2.1]{guerzhoy2010p}}]\label{lem coeff bound}
If $R=\sum a_{n}q^{n}$ is an element of $M_{k}^{!}(\Gamma_{0}(N), \mathbb{Q})$, then we have that
\begin{align*}
\inf_{n \in \mathbb{N}}v_{p}(a_{n}) >-\infty.
\end{align*}
\end{lem}

\begin{proposition}[\cite{bringmann2012mock}]\label{prop. lim=0}
If $p$ is inert in $\mathcal{O}_{K}$ and $p \nmid N$, then we have that 
\begin{align*}
\lim _{m\rightarrow \infty }\dfrac{a_{D^{k-1}(F)}( p^{2m}) }{\beta ^{2m}}=0.
\end{align*} 
\end{proposition}

\begin{proof}
Since $\chi_{K}(p^{2m})=1$ holds for all $m \in \mathbb{N}$, $a^{+}(p^{2m})$ is the $p^{2m}$-th coefficient of $\dfrac{1}{2}\left(F+F\otimes \chi_{K}\right)\otimes \chi_{K}$. From Lemma \ref{lemma b} and Lemma \ref{lem coeff bound}, we have that
\begin{align*}
v_{p}(a_{D^{k-1}(F)}(p^{2m})) \geq 2m(k-1)-A
\end{align*}
for some $A$.
We note that $v_{p}(\beta)=\dfrac{k-1}{2}$ because $\beta$ is a root of $X^{2}+p^{k-1}$.
Hence we conclude that $\lim _{m\rightarrow \infty }\dfrac{a_{D^{k-1}(F)}( p^{2m}) }{\beta ^{2m}}=0$.
\end{proof}

\begin{thm}[{\cite[Theorem 1.3]{bringmann2012mock}}]\label{def alpha_g}
Suppose that $p \nmid N$ and  $p$ is inert in $\mathcal{O}_{K}$.
Then there exists a unique $p$-adic constant $\alpha_{g} \in \mathbb{C}_{p}$ such that $\widetilde{\mathcal{F}}_{\alpha_{g}}$ is a $p$-adic modular form of weight $2-k$ and level $pN$. Furthermore, $\alpha_{g}$ is given by the $p$-adic limit 
\begin{align*}
\alpha_{g}=\displaystyle\lim _{m\rightarrow \infty }\dfrac{a_{D^{k-1}(F)}(p^{2m+1}) }{\beta ^{2m}}.
\end{align*}
\end{thm}

In addition, L. Candelori and F. Castella \cite{candelori2017geometric} showed that $\widetilde{\mathcal{F}}_{\alpha_{g}}$ is an overconvergent modular form by using geometric methods.

\section{Formal group associated to an elliptic curve}
In this section, we review the formal group associated to an elliptic curve. We refer to \cite{bannai2017radius} \cite{kobayashi2003iwasawa} as main references.
Let $g \in S_{2}( \Gamma _{0}( N), \mathbb{Z} )$ be a normalized newform with complex multiplication by an imaginary quadratic field $K$.
We define the period lattice $\Lambda$ of $g$ by
\begin{align*}
\Lambda:=\left\{ \int ^{\gamma(\infty)}_{\infty}2\pi ig(\tau) d\tau \mid \gamma \in \Gamma _{0}(N) \right\}.
\end{align*}
Let $g_{2}, g_{3}$ be rational numbers associated to $\Lambda$ defined by
\begin{align*}
g_{2}(\Lambda):=60\sum _{\lambda \in \Lambda \backslash \left\{ 0\right\} }\dfrac{1}{\lambda^{4}}, \quad g_{3}(\Lambda):=140\sum _{\lambda \in \Lambda \backslash \left\{ 0\right\} }\dfrac{1}{\lambda ^{6}}.
\end{align*}
 Then the elliptic curve $E(\mathbb{C}) := \mathbb{C}/\Lambda$ has a Weierstrass model defined over $\mathbb{Q}$
\begin{align}\label{shiki}
E : y^{2}=4x^{3}-g_{2}(\Lambda)x-g_{3}(\Lambda).
\end{align} 
It is known that the Weierstrass model \eqref{shiki} is defined over $\mathbb{Z}_{p}$ for a prime number $p \geq 5$. (cf.  \cite[Section 2.14]{cremona1997algorithms}.)
Let $\widehat{E}$ be the formal group associated to this Weierstrass model \eqref{shiki} with parameter $t=-2x/y$.
Let $p$ be inert in $\mathcal{O}_{K}$ and $p \geq 5$. We assume that the Weierstrass model \eqref{shiki} has good reduction at $p$.

For a one-dimensional formal group $\mathcal{F}$ defined over $\mathbb{Z}_{p}$, we define the space $H^{1}_{\text{cris}}(\widehat{E}/\mathbb{Z}_{p})$ following the approach in \cite{bannai2017radius}.
We denote the addition of $\mathcal{F}$ by $\oplus$ and define two spaces $\mathcal{Z}^{1}(\mathcal{F})$, $\mathcal{B}^{1}(\mathcal{F})$ by
\begin{align*}
&\mathcal{Z}^{1}(\mathcal{F}):=\{ \omega \in \widehat{\Omega}^{1}_{\mathbb{Z}_{p}[\![X]\!]/\mathbb{Z}_{p}} \mid F_{\omega}(X \oplus Y)-F_{\omega}(X)-F_{\omega}(Y) \in \mathbb{Z}_{p}[\![X]\!] \}
\\
&\mathcal{B}^{1}(\mathcal{F}):=\{df \in \widehat{\Omega}^{1}_{\mathbb{Z}_{p}[\![X]\!]/\mathbb{Z}_{p}} \mid f(X) \in \mathbb{Z}_{p}[\![X]\!] \}
\end{align*}
where $F_{\omega}$ is the formal primitive function of $\omega$. We say that $\omega \in \widehat{\Omega}^{1}_{\mathbb{Z}_{p}[\![X]\!]/\mathbb{Z}_{p}}$ is a differential form of the second kind if $\omega \in \mathcal{Z}^{1}(\mathcal{F})$. We define the space $H^{1}_{\text{cris}}(\mathcal{F}/\mathbb{Z}_{p})$ by
\begin{align*}
H^{1}_{\text{cris}}(\mathcal{F}/\mathbb{Z}_{p}):=\mathcal{Z}^{1}(\mathcal{F})/\mathcal{B}^{1}(\mathcal{F}).
\end{align*}
Let $\varphi$ be an endomorphism of $\mathbb{Q}_{p}[\![X]\!]$ defined by $f(X) \mapsto f(X^{p})$. It is known that $\varphi$ acts on $H^{1}_{\text{cris}}(\mathcal{F}/\mathbb{Z}_{p})$ by $\varphi(\sum_{i}f_{i}dg_{i})=\sum_{i}\varphi(f_{i})d\varphi(g_{i})$. (cf. \cite[Lemma 2.1]{bannai2017radius}.)

\begin{proposition}[{\cite[Proposition 2.3]{bannai2017radius}}]\label{prop. cris basis}
Let $\omega$ be an invariant differential form and $h$ be the height of the formal group $\mathcal{F}$. We put $\omega^{\ast}:=p^{-1}\varphi(\widehat{\omega})$.
Then $H^{1}_{\text{cris}}(\widehat{E}/\mathbb{Z}_{p})$ is a free $\mathbb{Z}_{p}$-module of rank $h$ with basis $\varphi^{i}\omega^{\ast} (0 \leq i \leq h-1)$. 
\end{proposition}

We return to the case $\mathcal{F}=\widehat{E}$. We denote the formal logarithm of $\widehat{E}$ by $\lambda(t)$ such that $\lambda'(0)=1$. We define an invariant differential form $\widehat{\omega}(t)$ and a differential form $\widehat{\eta}(t)$ by
\begin{align*}
\widehat{\omega}(t):=\dfrac{dx}{y}, \quad \widehat{\eta}(t):=x\dfrac{dx}{y}.
\end{align*}
As a formal power series, we have that
\begin{align*}
\widehat{\omega}(t)=dz\mid_{z=\lambda(t)}, \quad \widehat{\eta}(t)=\wp(\Lambda, z)dz\mid_{z=\lambda(t)}.
\end{align*}
\begin{lem}[{\cite[Lemma 3.4]{bannai2017radius}}]\label{eta_0}
 Let $\widehat{\eta}_{0}(t)$ be the differential form given by
 \begin{align*}
 \widehat{\eta}_{0}(t):=\widehat{\eta}(t)-\dfrac{dt}{t^{2}}.
 \end{align*}
 Then the differential form $\widehat{\eta}_{0}(t)$ is of the second kind.
\end{lem}

If $p$ is inert in $\mathcal{O}_{K}$, 
then $E$ has supersingular reduction, and the height of the formal group $\widehat{E}/\mathbb{Z}_{p}$ is equal to $2$.
This implies that $\{\widehat{\omega}(t), \omega^{\ast}(t)\}$ is a basis of $H^{1}_{\text{cris}}(\widehat{E}/\mathbb{Z}_{p})$ by Proposition \ref{prop. cris basis}.
The next theorem is the key to determining the $p$-adic valuation of $\alpha_{g}$.

\begin{thm}[{\cite[Corollary 3.8]{bannai2017radius}}]\label{thm A_{2} unit}
We write $\widehat{\eta}_{0}(t) \in H^{1}_{\text{cris}}(\widehat{E}/\mathbb{Z}_{p})$ by 
\begin{align*}
\widehat{\eta}_{0}(t)=A_{1}\widehat{\omega}(t)+A_{2}\omega^{\ast}(t) \text{ in } H^{1}_{\text{cris}}(\widehat{E}/\mathbb{Z}_{p}).
\end{align*}
Then the coefficient $A_{2}$ is a $p$-adic unit.
\end{thm}

We define the submodule $\mathcal{P}$ of $\mathbb{Q}_{p}[\![s]\!]$ by
\begin{align*}
\mathcal{P}:=\{ \sum ^{\infty }_{n=1}a_{n}s^{n} \mid na_{n} \in \mathbb{Z}_{p}  \text{ for all } n\}.
\end{align*}
\begin{dfn}\label{a}
 Let $\varphi$ be the endomorphism of $\mathcal{P}$ satisfying 
 \begin{align*}
 \varphi(\sum ^{\infty }_{n=1}a_{n}s^{n}):=\sum ^{\infty }_{n=1}a_{n}\left((1+s)^{p}-1\right)^{n}
 \end{align*}
 We define the endomorphism $u(\varphi)$ for $u(X)=\sum ^{m }_{n=1}b_{n}X^{n} \in \mathbb{Z}_{p}[X]$ by
 \begin{align*}
 u(\varphi)(f(s)):=\sum ^{m }_{n=1}b_{n}f\left(\varphi^{(n)}(s)\right).
 \end{align*}
\end{dfn}

\begin{lem}[{\cite[Lemma 8.1]{kobayashi2003iwasawa}}]
Let $\varphi$ be the endomorphism in Definition \ref{a}. Then, the action of $\varphi$ induces the Frobenius morphism 
\begin{align*}
\varphi : \mathcal{P}/p\mathbb{Z}_{p}[\![s]\!] \rightarrow \mathcal{P}/p\mathbb{Z}_{p}[\![s]\!].
\end{align*}
In other words, we have that
\begin{align*}
\varphi(\sum ^{\infty }_{n=1}a_{n}s^{n}) \equiv \sum ^{\infty }_{n=1}a_{n}s^{pn} \text{ mod } p\mathbb{Z}_{p}[\![s]\!]
\end{align*}
for all $\sum ^{\infty }_{n=1}a_{n}s^{n} \in \mathcal{P}/\mathbb{Z}_{p}[\![s]\!]$.
\end{lem}

\begin{dfn}[{\cite[Definition 8.2]{kobayashi2003iwasawa}}]
Let $f(s) \in \mathcal{P}$ and $u(X) \in \mathbb{Z}_{p}[X]$.
If $f(s)$ satisfies 
\begin{align*}
u(\varphi)(f) \equiv u(F)(f) \equiv 0 \text{ mod } p\mathbb{Z}_{p}[\![s]\!]
\end{align*}
and $f'(0)=1$, we say that $f(s)$ is of the Honda type $u$. 
\end{dfn}

\begin{proposition}\label{G type}
The formal power series $E_{g}(q)$ is of the Honda type $X^{2}+p$.
\end{proposition}  

\begin{proof}
We have that
\begin{align*}
(\varphi^{2}+p)E_{g}(q)&=\sum _ { n = 1 } ^ { \infty } \frac { a _ { n } } { n } q ^ { p ^ { 2 } n } + \sum _ { k = 1 } ^ { \infty } \frac { p a _ { n } } { n } q ^ { n }
\\
&= \sum ^{\infty }_{n=1}\left( \dfrac{a_{n}}{n}+\dfrac{a_{p^{2}n}}{pn}\right) q^{p^{2}n}+\sum _{p \nmid n}\dfrac{a_{pn}}{n}q^{pn}+\sum _{p \nmid n}p\dfrac{a_{n}}{n}q^{n},
\end{align*}
and 
\begin{align*}
a_{pn}+pa_{n/p}=0
\end{align*}
holds because $T_{2, p}g=a_{g}(p)g=0$. Therefore, we obtain that
\begin{align*}
\dfrac{a_{n}}{n}+\dfrac{a_{p^{2}n}}{pn}=\dfrac{a_{n}}{n}-\dfrac{pa_{n}}{pn}=0
\end{align*}
 for all $n$ and
\begin{align*}
\dfrac{a_{pn}}{n}=-\dfrac{a_{n/p}}{n}=0
\end{align*}
for all $p \nmid n$.
From these results, we conclude that
\begin{align*}
(\varphi^{2}+p)E_{g}(q)=\sum _{p \nmid n}p\dfrac{a_{n}}{n}q^{n} \in p\mathbb{Z}_{p}[\![q]\!].
\end{align*}
\end{proof}
We define an Eisenstein polynomial to a polynomial 
\begin{align*}
u(X)=a_{n}X^{n}+a_{n-1}X^{n-1}+ \cdots +a_{0}
\end{align*}
satisfying that $a_{n}$ is a unit, $p \mid a_{i}$ for all $i \not=n$ and $a_{0}=p$. 

\begin{thm}[{\cite[Theorem 8.3]{kobayashi2003iwasawa}}]\label{thm exist formal group}
Let $u(X)$ be an Eisenstein polynomial and $f \in \mathcal{P}$ be of the Honda type $u$. Then there exists a formal group over $\mathbb{Z}_{p}$ with the logarithm $f$.
\end{thm}

Combining Proposition \ref{G type} and Theorem \ref{thm exist formal group}, there exists a formal group $\mathcal{G}$ over $\mathbb{Z}_{p}$ whose logarithm is equal to $E_{g}(q)$. Next, we show that $\mathcal{G}$ is isomorphic to $\widehat{E}$.

\begin{thm}[{\cite[Theorem 8.4]{kobayashi2003iwasawa}}]\label{E type}
Let $\log_{\widehat{E}}$ be the the logarithm of $\widehat{E}$. 
Then $\log_{\widehat{E}}$ is of the Honda type $X^{2}+p$.
\end{thm}

\begin{thm}[{\cite[Theorem 8.3]{kobayashi2003iwasawa}}]\label{isom}
Let $\mathcal{F}$ and $\mathcal{F'}$ be the formal groups over $\mathbb{Z}_{p}$. If their logarithm are of the same Honda type $u$, then $\exp_{\mathcal{F}'} \circ \log_{\mathcal{F}}$ is a formal power series defined over $\mathbb{Z}_{p}$ and gives an isomorphism $\mathcal{F} \rightarrow \mathcal{F}'$.
\end{thm}

\begin{thm}\label{thm isom of formal group}
Let $f(q)$ be the formal power series defined by $f(q)=\exp_{\widehat{E}} \circ \log_{\mathcal{G}}$. Then $f(q)$ is an element of $\mathbb{Z}_{p}[\![q]\!]$ and gives an isomorphism from $\mathcal{G}$ to $\widehat{E}$.
\end{thm}

\begin{proof}
This theorem is followed by Proposition \ref{G type}, Theorem \ref{E type} and Theorem \ref{isom}.
\end{proof}

\section{The $p$-adic properties of weight $0$ mock modular forms}\label{sec p-adic mock 0}
In this section, we review of \cite{guerzhoy2014zagier} in the case that  $g \in S_{2}( \Gamma _{0}( N), \mathbb{Z} )$ is a normalized newform with complex multiplication by an imaginary quadratic field $K$. We refer to \cite{guerzhoy2014zagier} \cite{alfes2015weierstrass} as main references.
 We have the map $\phi$ from $X_{0}(N)$ to $\mathbb{C}/\Lambda$ defined by
\begin{align*}
\phi : \tau \mapsto -\int^{i\infty}_{\tau} 2\pi ig(\tau')d\tau'=\sum_{n \geq 1}\dfrac{a_{g}(n)} {n}q^{n} \text{ mod } \Lambda
\end{align*}
where $\Lambda$ is the period latticeof $g$ and $q=e^{2\pi i \tau}$.
For $z \in \mathbb{C}$, we define the Weierstrass $\zeta$-function associated with $\Lambda$ by
\begin{align*}
\zeta ( \Lambda ,z) :=\dfrac{1}{z}+\sum _{\lambda \in \Lambda \backslash \left\{ 0\right\} }\left( \dfrac{1}{z-\lambda}+\dfrac{z}{\lambda^{2}} + \dfrac{1}{\lambda}\right).
\end{align*}
It is well-known that $\zeta ( \Lambda ,z)$ is a meromorphic function on $\mathbb{C}$ and not $\Lambda$-periodic.
However, we could modify $\zeta ( \Lambda ,z)$ to a $\Lambda$-periodic function. Let $a(\Lambda)$ be the area of the fundamental parallelogram of $\Lambda$, and we define the complex number $\mathbb{S}(\Lambda)$ is the Eisenstein number of weight $2$ by
\begin{align*}
\mathbb{S}(\Lambda):=\lim _{s\rightarrow 0}\sum _{\lambda \in \Lambda \backslash \left\{ 0\right\} }\dfrac{1}{\lambda^{2}\left| \lambda\right| ^{2s}}
\end{align*}
Then, it is known that the function 
\begin{align*}
R(z):=\zeta ( \Lambda ,z) -\mathbb{S}(\Lambda)z - \dfrac{\pi}{a(\Lambda)}\overline{z}
\end{align*}
is $\Lambda$-periodic.
Therefore, we can consider the map $\mathcal{N}:=R \circ \phi$ from $X_{0}(N)$ to $\mathbb{C}$. The function $\mathcal{N}(\tau)$ has a certain harmonic property. We define the $q$-expansion of $\mathcal{N}(\tau)$ by $\mathcal{N}(q):=R(E_{g}(q))$.
The $q$-series $\mathcal{N}(q)$ decomposes into two sums $\mathcal{N}(q)=\mathcal{N}^{+}(q)+\mathcal{N}^{-}(q)$ where
\begin{align*}
\mathcal{N}^{+}(q):=\zeta ( \Lambda ,E_{g}(q)) -\mathbb{S}(\Lambda)E_{g}(q), \quad \mathcal{N}^{-}(q):=- \dfrac{\pi}{a(\Lambda)}\overline{E_{g}(q)}
\end{align*}
Let $F$ be a harmonic Maass form that is good for $g$ and $C_{g}$ be the degree of the map $\phi : X_{0}(N) \rightarrow \mathbb{C}/ \Lambda$.
We show that $\mathcal{N}(q)-C_{g}F(q) \in \mathbb{Z}_{p}(\!(q)\!) \otimes \mathbb{Q}_{p}$ for all inert prime $p \nmid N$.

\begin{lem}[{\cite[Proposition 2]{guerzhoy2014zagier}}]
We have that
 \begin{align*}
 \mathcal{N}(q)-C_{g}F(q)=\mathcal{N}^{+}(q)-C_{g}F^{+}(q) \in \mathbb{Q}(\!(q)\!).
 \end{align*}
\end{lem}

We denote the $q$-expansion of $\mathcal{N}^{+}(q)-C_{g}F^{+}(q)$  by $\sum b(n)q^{n}$.

\begin{lem}\label{lem. con. rad.}
The radius of convergence of the non-principal part of $\mathcal{N}^{+}(q)-C_{g}F^{+}(q)$ is at least $1$.
\end{lem}

\begin{proof}
It is sufficient to show that 
\begin{align*}
D(\sum_{n \geq 0} b(n)q^{n})=\sum_{n \geq 0} nb(n)q^{n} \in \mathbb{Z}_{p}(\!(q)\!) \otimes \mathbb{Q}_{p}.
\end{align*}
In other words, we show that the coefficients $nb(n)$ have bounded denominators.
Firstly, we show that $D(\mathcal{N}^{+}(q)) \in \mathbb{Z}_{p}(\!(q)\!) \otimes \mathbb{Q}_{p}$.
Since 
\begin{align*}
D(\mathcal{N}^{+}(q))=-\wp(\Lambda, E_{g}(q))g(q)-\mathbb{S}(\Lambda)g(q),
\end{align*}
it is enough to show that $\wp(\Lambda, E_{g}(q)) \in \mathbb{Z}_{p}(\!(q)\!) \otimes \mathbb{Q}_{p}$. By Lemma \ref{eta_0} and Theorem \ref{thm isom of formal group}, we have that 
\begin{align*}
f^{\ast}(\widehat{\eta}_{0}(t))=\wp(\Lambda, E_{g}(q))\dfrac{g(q)}{q}+\dfrac{d}{dq}\left(\dfrac{1}{f(q)}\right) \in H^{1}_{\text{cris}}(\mathcal{G}/\mathbb{Z}_{p}).
\end{align*}
Therefore, we obtain that $\wp(\Lambda, E_{g}(q)) \in \mathbb{Z}_{p}(\!(q)\!) \otimes \mathbb{Q}_{p}$ because $f(q) \in \mathbb{Z}_{p}[\![q]\!]^{\times}$.
Next, we show that $D(F^{+}(q)) \in \mathbb{Z}_{p}(\!(q)\!) \otimes \mathbb{Q}_{p}$. It is known that the differential operator $D$ gives the map 
\begin{align*}
D:H_{0} (\Gamma_{0}(N)) \rightarrow M_{2}^{!}(\Gamma_{0}(N)). 
\end{align*}
Hence $D(F^{+}) \in M_{2}^{!}(\Gamma _{0}(N) , \mathbb{Q})$ holds by Theorem \ref{Thm alg of cof}. From Lemma \ref{lem coeff bound}, we have that $D(F^{+}(q)) \in \mathbb{Z}_{p}(\!(q)\!) \otimes \mathbb{Q}_{p}$, and we complete the proof.

\end{proof}

\begin{thm}\label{thm N-CF bounded}
We have that
 \begin{align}\label{eq. bounded}
 \mathcal{N}(q)-C_{g}F(q)=\mathcal{N}^{+}(q)-C_{g}F^{+}(q) \in \mathbb{Z}_{p}(\!(q)\!) \otimes \mathbb{Q}_{p}
 \end{align}
 for all inert prime $p \nmid N$.
\end{thm}

P. Guerzhoy showed that \eqref{eq. bounded} holds for almost all prime number $p$. (cf. \cite[Proposition 2]{guerzhoy2014zagier}.) We refine the proof of \cite[Proposition 2]{guerzhoy2014zagier}, and we show that \eqref{eq. bounded} holds for all inert prime $p \nmid N$.

\begin{proof}
By the proof of {\cite[Proposition 2]{guerzhoy2014zagier}}, $\mathcal{N}^{+}(q)-C_{g}F^{+}(q)$ is a $\Gamma_{0}(N)$-invariant meromorphic function whose pole $\tau_{i} \in \mathbb{H}$ satisfies that $j(\tau_{i}) \in \bar{\mathbb{Q}}$. Let $K$ be a number field contains all values $j(\tau_{i})$ and $\pi$ be a uniformizer of $\mathcal{O}_{K}$. Then there exists a monic polynomial $G(X) \in K[X]$ and $m \in \mathbb{N}$ such that 
\begin{align*}
\Delta(q)^{m}G(j(q))({N}^{+}(q)-C_{g}F^{+}(q)) \in S_{k}(\Gamma_{0}(N), K).
\end{align*}
Therefore, 
\begin{align*}
\pi^{k_{1}}q^{l_{1}}\Delta(q)^{m}G(j(q))({N}^{+}(q)-C_{g}F^{+}(q)) \in \mathcal{O}_{K}[\![q]\!]
\end{align*}
whose all coefficients are not in $\pi\mathcal{O}_{K}$ for some $k_{1}, l_{1} \in \mathbb{N}$.
By Weierstrass preparation theorem, there exists a distinguished polynomial $P(q)$ and a unit $u_{1}(q) \in \mathcal{O}_{K}[\![q]\!]$ such that 
\begin{align}\label{eq. u_1}
\pi^{k_{1}}q^{l_{1}}\Delta(q)^{m}G(j(q))({N}^{+}(q)-C_{g}F^{+}(q))=P(q)u_{1}(q)
\end{align}
holds.
Similarly, there exists a distinguished polynomial $h(q)$ that is coprime to $P(q)$ and a unit $u_{2}(q) \in \mathcal{O}_{K}[\![q]\!]$ such that
\begin{align}\label{eq. u_2}
\pi^{k_{2}}q^{l_{2}}\Delta(q)^{m}G(j(q))=h(q)u_{2}(q)
\end{align}
holds for some $k_{2}, l_{2} \in \mathbb{N}$. Combining \eqref{eq. u_1} and \eqref{eq. u_2}, we obtain that
\begin{align}\label{eq. dis. pol.}
{N}^{+}(q)-C_{g}F^{+}(q)=\pi^{k_{2}-k_{1}}q^{l_{2}-l_{1}}\dfrac{P(q)}{h(q)}u_{1}(q)u_{2}(q)^{-1}.
\end{align}
Hence, it is enough to show that $h(q)$ has no nonzero root to complete the proof.
Suppose that $h(q)$ has a nonzero root. We denote the nonzero roots of $h(q)$ by $\alpha_{i}$. Since $h(q)$ is a distinguished polynomial, each $\alpha_{i}$ satisfies that $|\alpha_{i}| < 1$.
Therefore, the radius of convergence of the non-principal part of \eqref{eq. dis. pol.} is equal to $\min_{i} |\alpha_{i}| < 1$ . This is a contradiction to Lemma \ref{lem. con. rad.}.
\end{proof}

\begin{proposition}[{\cite[Proposition 5]{guerzhoy2014zagier}}]\label{prop. ramuda}
There exist unique $p$-adic numbers $\lambda_{p}, \mu_{p} \in \mathbb{Q}_{p}$ such that
\begin{align*}
F^{+}(q)-\lambda_{p}E_{g}(q)-\mu_{p}E_{g\mid V_{p}}(q) \in \mathbb{Z}_{p}(\!(q)\!) \otimes \mathbb{Q}_{p}
\end{align*}
for all prime number $p \nmid N$. 
\end{proposition}

We give a characterization of the $p$-adic constant $\alpha_{g}$ by using Proposition \ref{prop. ramuda}.

\begin{thm}\label{mu=alpha}
Let $p$ be a prime number satisfying that inert in $\mathcal{O}_{K}$ and $p \nmid N$.
If there exist $p$-adic numbers $\lambda_{p}, \mu_{p} \in \mathbb{Q}_{p}$ such that
\begin{align*}
F^{+}(q)-\lambda_{p}E_{g}(q)-\mu_{p}E_{g\mid V_{p}}(q) \in \mathbb{Z}_{p}(\!(q)\!) \otimes \mathbb{Q}_{p},
\end{align*}
then $\lambda_{p}=0$ and $\mu_{p}=\alpha_{g}$. 
\end{thm}

\begin{proof}
From Proposition \ref{prop. ramuda}, it is sufficient to show that $\lambda_{p}=0, \mu_{p}=\alpha_{g}$ when $p$ is inert in $\mathcal{O}_{K}$.
We assume that $p$ is inert in $\mathcal{O}_{K}$ .
It is known that there exists a weakly holomorphic modular form $R_{p} \in M_{0}^{!}( \Gamma _{0}(N), \mathbb{Q})$ such that
\begin{align*}
F^{+}\mid T_{0, p}=R_{p}.
\end{align*}
(See the proof of \cite[Theorem 1.3]{bruinier2008differential}.)
Let $\beta, \beta'$ be the roots of the polynomial $X^{2}-a_{g}(p)X+p=X^{2}+p$. 
We define two $q$-series $\mathcal{G}, \mathcal{G}'$ by
\begin{align*}
\mathcal{G}:=F^{+}-\beta^{-1}F^{+}\mid V_{p}, \quad \mathcal{G}':=F^{+}-\beta'^{-1}F^{+}\mid V_{p}.
\end{align*}
P. Guerzhoy \cite{guerzhoy2014zagier} showed that
$
\lim _{l\rightarrow \infty }\beta'^{l}\mathcal{G}\mid U_{p}^{l}$ and  $\lim _{l\rightarrow \infty }\beta^{l}\mathcal{G}'\mid U_{p}^{l}$
are convergent, and there exist $u, v \in \mathbb{Q}_{p}(\sqrt{-p})$ such that
\begin{align*}
\lim _{l\rightarrow \infty }\beta'^{l}\mathcal{G}\mid U_{p}^{l}=u\left(E_{g}(q)-\beta^{-1}E_{g}(q^{p})\right), \quad \lim _{l\rightarrow \infty }\beta^{l}\mathcal{G}'\mid U_{p}^{l}=v\left(E_{g}(q)-\beta'^{-1}E_{g}(q^{p})\right).
\end{align*}
In the proof of \cite[Proposition 5]{guerzhoy2014zagier}, P. Guerzhoy \cite{guerzhoy2014zagier} gives that 
\begin{align*}
\lambda_{p}=\dfrac{u+v}{2}, \quad \mu_{p}=p\dfrac{v-u}{2\beta}=\dfrac{\beta(u-v)}{2}.
\end{align*}
We show that $u=\beta^{-1}\alpha_{g}, v=-u$.
Let $a_{l}(1)$ be the first Fourier coefficient of $\beta'^{l}\mathcal{G}\mid U_{p}^{l}$.
Then we have that
\begin{align*}
a_{l}(1)=\beta'^{l}\left(a^{+}(p^{l})-\beta^{-1}a^{+}(p^{l-1})\right)=\dfrac{a_{D(F)}(p^{l})}{\beta^{l}}+\dfrac{a_{D(F)}(p^{l-1})}{\beta^{l-1}}
\end{align*}
where $a_{D(F)}(n)=na^{+}(n)$ is the $n$-th Fourier coefficients of $D(F)$.
By Theorem \ref{def alpha_g} and Proposition \ref{prop. lim=0}, we obtain that
\begin{align*}
u=\lim _{l\rightarrow \infty }a_{l}(1)=\lim _{l\rightarrow \infty }a_{2l}(1)=\beta^{-1}\alpha_{g}.
\end{align*}
Similarly, let $b_{l}(1)$ be the first Fourier coefficient of $\beta'^{l}\mathcal{G}\mid U_{p}^{l}$.
Then we have that
\begin{align*}
b_{2l}(1)=\beta^{2l}\left(a^{+}(p^{2l})-\beta'^{-1}a^{+}(p^{2l-1})\right)=\dfrac{a_{D(F)}(p^{2l})}{\beta^{2l}}-\dfrac{a_{D(F)}(p^{2l-1})}{\beta^{2l-1}}.
\end{align*}
Therefore, 
\begin{align*}
v=\lim _{l\rightarrow \infty }b_{l}(1)=\lim _{l\rightarrow \infty }b_{2l}(1)=-\beta^{-1}\alpha_{g}=-u
\end{align*}
holds.
Hence, we conclude that
\begin{align*}
\lambda_{p}=0, \quad \mu_{p}=\dfrac{\beta(u-v)}{2}=\beta u=\alpha_{g}.
\end{align*}
\end{proof}

\section{Proof of Main Theorem}\label{proof section}
Let $g \in S_{2}(\Gamma_{0}(N), \mathbb{Z})$ be a normalized newform with complex multiplication by an imaginary quadratic field $K$ and $p$ be inert in $\mathcal{O}_{K}$ and $p \geq 5$. We assume that the Weierstrass model \eqref{shiki} has good reduction at $p$. 

We calculate a basis of $H^{1}_{\text{cris}}(\mathcal{G}/\mathbb{Z}_{p})$.
Let $\phi$ be the Frobenius morphism on $H^{1}_{\text{cris}}(\mathcal{G}/\mathbb{Z}_{p})$ then we have that
\begin{align*}
\omega_{g}^{\ast}(q)&=p^{-1}\phi\left(\sum ^{\infty }_{n=1}a_{g}(n)q^{n-1}dq\right)
\\
&=\dfrac{d}{dq}\left(\sum _{n>0}\dfrac{a_{g}(n)}{pn} q^{pn}\right)
\\
&=\dfrac{d}{dq}E_{g\mid V_{p}}(q).
\end{align*}

Let $f(q)$ be the formal power series defined by $f(q)=\exp_{\widehat{E}} \circ \log_{\mathcal{G}}$.
Next, we calculate the image of $\widehat{\eta}_{0}(t)$ by $f^{\ast}$.

\begin{lem}
Let $\widehat{\eta}_{0}(t)$ be the differential form given by
 \begin{align*}
 \widehat{\eta}_{0}(t):=\widehat{\eta}(t)-\dfrac{dt}{t^{2}} \in H^{1}_{\text{cris}}(\widehat{E}/\mathbb{Z}_{p}).
 \end{align*}
 Then the equation
\begin{align*}
f^{\ast}(\widehat{\eta}_{0}(t))=\dfrac{d}{dq}\left(-\zeta(\Lambda, E_{g}(q))+\dfrac{1}{f(q)}\right)
\end{align*}
holds.
\end{lem}

\begin{proof}
Since $\widehat{\eta}(t)=\wp(\Lambda, z)dz \mid_{z=\lambda(t)}$, we have that
\begin{align*}
f^{\ast}(\widehat{\eta}(t))
&=\wp(\Lambda, z)dz \mid_{z=\lambda(f(q))}
\\&=\wp(\Lambda, z)dz \mid_{z=E_{g}(q)}
\\&=-\dfrac{d}{dq}\zeta(\Lambda, E_{g}(q)).
\end{align*}
It is clear that $f^{\ast}(\dfrac{dt}{t^{2}})=-\dfrac{d}{dq}\left(\dfrac{1}{f(q)}\right)$, and the claim follows from this. 
\end{proof}
Since the isomorphism $f^{\ast} : H^{1}_{\text{cris}}(\widehat{E}/\mathbb{Z}_{p}) \rightarrow H^{1}_{\text{cris}}(\mathcal{G}/\mathbb{Z}_{p})$ is compatible with the Frobenius morphism, there is a $p$-adic unit $e$ such that 
\begin{align}\label{e}
f^{\ast}(\omega^{\ast})=e\omega_{g}^{\ast}=e\dfrac{d}{dq}E_{g\mid V_{p}}(q).
\end{align}
We denote $\widehat{\eta}_{0}(t)$ by
\begin{align*}
\widehat{\eta}_{0}(t)=A_{1}\widehat{\omega}(t)+A_{2}\omega^{\ast}(t) \text{ in } H^{1}_{\text{cris}}(\widehat{E}/\mathbb{Z}_{p}).
\end{align*}
Then, we obtain that
\begin{align}\label{eq daizi}
f^{\ast}(\widehat{\eta}_{0}(t))=\dfrac{d}{dq}\left(-\zeta(\Lambda, E_{g}(q))+\dfrac{1}{f(q)}\right)=A_{1}\dfrac{d}{dq}E_{g}(q)+eA_{2}\dfrac{d}{dq}E_{g\mid V_{p}}(q).
\end{align}
From this equation \eqref{eq daizi}, there exists a formal power series $F(q) \in \mathbb{Z}_{p}[\![q]\!]$ such that
\begin{align*}
-\zeta(\Lambda, E_{g}(q))=A_{1}E_{g}(q)+eA_{2}E_{g\mid V_{p}}(q)+\dfrac{1}{f(q)}+F(q).
\end{align*}
Since $f(q) \in \mathbb{Z}_{p}[\![q]\!]^{\times}$, we conclude that
\begin{align}\label{zeta bounded}
\zeta(\Lambda, E_{g}(q))+A_{1}E_{g}(q)+eA_{2}E_{g\mid V_{p}}(q) \in \mathbb{Z}_{p}(\!(q)\!) \otimes \mathbb{Q}_{p}.
\end{align}
We note that $eA_{2}$ is a $p$-adic unit by \eqref{e} and Theorem \ref{thm A_{2} unit}.
We prove the main theorem.

\begin{thm}\label{main_THM}
Let $g \in S_{2}(\Gamma_{0}(N))$ be a normalized newform with complex multiplication by an imaginary quadratic field $K$ with rational Fourier coefficients. We assume that $p \geq 5$ and the Weierstrass model \eqref{shiki} has good reduction at $p$.
If $p$ is inert in $\mathcal{O}_{K}$, then $C_{g}\alpha_{g}$ is a $p$-adic unit. 
\end{thm}

\begin{proof}
By Theorem \ref{thm N-CF bounded}, we have that 
 \begin{align*}
 \mathcal{N}(q)-C_{g}F(q)=\mathcal{N}^{+}(q)-C_{g}F^{+}(q) =\zeta ( \Lambda ,E_{g}(q)) -\mathbb{S}(\Lambda)E_{g}(q)-C_{g}F^{+}(q) \in \mathbb{Z}_{p}(\!(q)\!) \otimes \mathbb{Q}_{p}.
 \end{align*}
 Hence, we obtain that 
 \begin{align*}
 C_{g}F^{+}+\left(\mathbb{S}(\Lambda)+A_{1}\right)E_{g}(q)+eA_{2}E_{g\mid V_{p}}(q) \in \mathbb{Z}_{p}(\!(q)\!) \otimes \mathbb{Q}_{p}
 \end{align*}
 by using of \eqref{zeta bounded}.
 Taking into the account Theorem \ref{mu=alpha}, we obtain that
 \begin{align*}
 \alpha_{g}=-C_{g}^{-1}(eA_{2}).
 \end{align*}
 The claim follows from that $eA_{2}$ is a $p$-adic unit.
\end{proof}

\section*{Declarations}
\textbf{Conflict of interest} The author declares that he has no conflict of interest.

\bibliographystyle{amsplain}
\bibliography{References2}

\end{document}